\documentclass{conm-p-l}

\usepackage{amsfonts, amstext, amsmath, amsthm, amscd, amssymb}
\usepackage{epsfig, graphicx, psfrag}
\usepackage{color}
\usepackage[OT2,OT1]{fontenc} 

\newcommand\cyr{%
\renewcommand\rmdefault{wncyr}%
\renewcommand\sfdefault{wncyss}%
\renewcommand\encodingdefault{OT2}%
\normalfont
\selectfont}
\DeclareTextFontCommand{\textcyr}{\cyr}

\renewcommand{\setminus}{{\smallsetminus}}

\newcommand{\HH}{{\mathbb{H}}}
\newcommand{\RR}{{\mathbb{R}}}
\newcommand{\ZZ}{{\mathbb{Z}}}

\newcommand{\CC}{{\mathbb{C}}}

\newcommand{\rr}{{\mathbf{r}}}
\newcommand{\qq}{{\mathbf{q}}}
\newcommand{\vv}{{\mathbf{v}}}
\newcommand{\ww}{{\mathbf{w}}}

\newcommand{\ang}{{\mathcal{A}}}
\newcommand{\circang}{{\mathcal{SA}}}
\newcommand{\bdy}{{\partial}}
\newcommand{\vol}{{\mathcal{V}}}

\newcommand{\abs}[1]{{\left\vert #1 \right\vert}}
\newcommand{\thh}{{\mathrm{th}}}
\newcommand{\rank}{{\mathrm{rank}}}
\newcommand{\eps}{{\varepsilon}}
\newcommand{\lob}{{\textcyr{L}}}
\newcommand{\Imm}{{\mathrm{Im}}}
\newcommand{\Ree}{{\mathrm{Re}}}

\theoremstyle{plain}
\newtheorem{theorem}{Theorem}[section]

\newtheorem{lemma}[theorem]{Lemma}
\newtheorem{prop}[theorem]{Proposition}

\newtheorem*{namedtheorem}{\theoremname}
\newcommand{\theoremname}{testing}
\newenvironment{named}[1]{\renewcommand{\theoremname}{#1}\begin{namedtheorem}}{\end{namedtheorem}}

\theoremstyle{definition}
\newtheorem{define}[theorem]{Definition}

\newtheorem{example}[theorem]{Example}

\numberwithin{equation}{section}

\begin{document}
\title{From angled triangulations to hyperbolic structures}

\author{David Futer}
\address{Department of Mathematics, Temple University,
Philadelphia, PA 19122, USA}
\email{dfuter@temple.edu}
\thanks{Futer is supported in part by NSF Grant No. DMS--1007221.}

\author{Fran\c{c}ois Gu\'eritaud}
\address{Laboratoire Paul Painlev\'e, CNRS UMR 8524, Universit\'e de Lille 1, 59650 Ville\-neuve d'Ascq, France}
\email{Francois.Gueritaud@math.univ-lille1.fr}
\thanks{Gu\'eritaud is supported in part by the ANR program ETTT (ANR-09-BLAN-0116-01).}

\date{ \today}

\begin{abstract}
This survey paper contains an elementary exposition of Casson and Rivin's
technique for finding the hyperbolic metric on a $3$--manifold $M$ with toroidal boundary.
We also survey a number of applications of this technique.

The method involves subdividing $M$ into ideal tetrahedra and solving a system of gluing 
equations to find hyperbolic shapes for the tetrahedra. The gluing equations decompose into a
linear and non-linear part. The solutions to the linear equations form a convex polytope $\ang$. The 
solution to the non-linear part (unique if it exists) is a critical point of a certain \emph{volume functional} 
on this polytope. The main contribution of this paper is an elementary proof of Rivin's theorem
that a critical point of the volume functional on $\ang$ produces a complete hyperbolic structure 
on $M$.
\end{abstract}

\maketitle

\section{Introduction}\label{sec:intro}

Around 1980, William Thurston showed that ``almost every'' $3$--manifold $M$, whether closed or bounded, admits a complete hyperbolic metric \cite{thurston:survey}. When the boundary of $M$ consists of tori, this metric is unique up to isometry \cite{mostow:rigidity, prasad:rigidity}. Thurston introduced a method for finding this unique metric. The idea was to subdivide the interior of $M$ into \emph{ideal tetrahedra} (tetrahedra whose vertices are removed), and then give those tetrahedra hyperbolic shapes that glue up coherently in $M$. The shape of a hyperbolic ideal tetrahedron can be completely described by a single complex number, namely the cross--ratio of its four vertices on the sphere at infinity. Thurston wrote down a system of \emph{gluing equations} in those complex parameters, whose solution corresponds to the complete hyperbolic metric on the interior 
 of $M$ \cite{thurston:notes}.

The difficulty with Thurston's approach is that this non-linear system of equations is very difficult to solve in practice. Even proving the existence of a positively oriented solution, for a given triangulation, often turns out to be a daunting task.

In the 1990s, Andrew Casson and Igor Rivin discovered a powerful technique for solving Thurston's gluing equations. Their main idea (which builds on a result of Colin de Verdi{\`e}re \cite{cdv:variations}) was to separate the system into a \emph{linear part} and a \emph{non-linear part}. The linear part of the equations corresponds geometrically to a study of dihedral angles of the individual tetrahedra. In order for the tetrahedra to fit together in $M$, the dihedral angles at each edge of $M$ must sum to $2\pi$; in order for the tetrahedra to be positively oriented, all the angles must be positive. For a given triangulation $\tau$, the space of solutions to this linear system of equations and inequalities is a convex polytope $\ang(\tau)$, with compact closure $\overline{\ang(\tau)}$. A point of $\ang(\tau)$, corresponding to an assignment of tetrahedron shapes that satisfies the angle conditions, is called an \emph{angle structure} on $\tau$.  See Definition  \ref{def:angled-polytope} for details.

An angle structure on $\tau$ is weaker than a hyperbolic metric, because the linear equations that define $\ang(\tau)$ impose a strictly weaker condition on tetrahedra than do Thurston's gluing equations. Nevertheless, Casson proved

\begin{theorem}\label{thm:ang-hyperbolization}
Let $M$ be an orientable $3$--manifold with boundary consisting of tori, and let $\tau$ be an ideal triangulation of $M$. If $\ang(\tau) \neq \emptyset$, then $M$ admits a complete hyperbolic metric.
\end{theorem}


In other words, the existence of an angle structure implies the existence of a hyperbolic structure. Finding this hyperbolic structure requires solving the non-linear part of the gluing equations. To do this, Casson and Rivin introduced a \emph{volume functional} $\vol: \ang(\tau) \to \RR$, which assigns to every angle structure on $\tau$ the sum of the hyperbolic volumes of the tetrahedra in this structure. See Definition \ref{def:vol-functional} for a precise definition. As we shall see in Section \ref{sec:volume-max}, the function $\vol$ has a number of pleasant properties: it is smooth and strictly concave down on $\ang(\tau)$, and extends continuously to $\overline{\ang(\tau)}$. Casson and Rivin independently proved

\begin{theorem}\label{thm:volume-max}
Let $M$ be an orientable $3$--manifold with boundary consisting of tori, and let $\tau$ be an ideal triangulation of $M$. Then a point $p \in \ang(\tau)$ corresponds to a complete hyperbolic metric on the interior of $M$ if and only if $p$ is a critical point of the functional $\vol: \ang(\tau) \to \RR$.
\end{theorem}

All told, the Casson--Rivin program can be summarized as follows. Solving the linear part of Thurston's gluing equations produces the convex polytope $\ang(\tau)$. Solving this linear system is straightforward and algorithmic; by Theorem \ref{thm:ang-hyperbolization}, the existence of a solution implies that a hyperbolic structure also exists. By Theorem \ref{thm:volume-max}, solving the non-linear part of the gluing equations amounts to finding a critical point (necessarily a global maximum) of the functional $\vol: \ang(\tau) \to \RR$. In practice, the search for this maximum point can be accomplished by gradient--flow algorithms. It is worth emphasizing that a critical point of $\vol$ does not always exist: the maximum of $\vol$ over the compact closure $\overline{\ang(\tau)}$ will occur at a critical point in $\ang(\tau)$ if and only all the tetrahedra of $\tau$ are positively oriented in the hyperbolic metric on $M$.

\subsection{Where to find proofs}
Proofs of Theorem \ref{thm:ang-hyperbolization} are readily available in the literature. A nice account appears in Lackenby \cite[Corollary 4.6]{lackenby:surgery}, and a slightly more general result appears in Futer and Gu\'eritaud \cite[Theorem 1.1]{fg-arborescent}. Here is a brief summary of the argument. By Thurston's hyperbolization theorem \cite{thurston:survey}, the existence of a hyperbolic structure is equivalent to the \emph{non}-existence of essential spheres, disks, tori, and annuli in $M$. Any such essential surface $S$ can be moved into a \emph{normal form} relative to the triangulation $\tau$; this means that every component of intersection between $S$ and a tetrahedron is a disk in one of several combinatorial types. The dihedral angles that come with an angle structure permit a natural measure of complexity for normal disks, which mimics the area of hyperbolic polygons. As a result, one can show that every surface $S$ with non-negative Euler characteristic must have non-positive area, which means it is inessential. Hence $M$ is hyperbolic. 

By contrast, direct proofs of Theorem \ref{thm:volume-max} are harder to find. The standard reference for this result is Rivin \cite{rivin:volume}. However, the focus of Rivin's paper is somewhat different: he mainly studies the situation where $M$ is an ideal polyhedron, subdivided into ideal tetrahedra that meet at a single vertex. The notion of an angle space $\ang(\tau)$ and a volume functional $\vol$ still makes sense in this context, and Rivin's proof of the analogous result \cite[Lemma 6.12 and Theorem 6.16]{rivin:volume} extends (with some effort by the reader) to manifolds with torus boundary. 

One reference that contains the above formulation of Theorem \ref{thm:volume-max}, together with a direct proof, is Chan's undergraduate honors thesis \cite[Theorem 5.1]{chan-hodgson}. However, this thesis is not widely available. In addition, Chan's argument relies on a certain symplectic pairing introduced by Neumann and Zagier \cite{neumann-zagier}, a layer of complexity that is not actually necessary.

The central goal of this paper is to write down an elementary, self-contained proof of Theorem \ref{thm:volume-max}. The argument is organized as follows. In Section \ref{sec:gluing}, we recall the definition of Thurston's gluing equations. In Section \ref{sec:angle-struct}, we give a rigorous definition of the polytope of angle structures $\ang(\tau)$ and compute its dimension. In Section \ref{sec:lead-trail}, we introduce a natural set of tangent vectors, called \emph{leading--trailing deformations}, which span the tangent space $T_p \ang(\tau)$. In Section \ref{sec:volume-max}, we focus on the volume functional $\vol:\ang(\tau) \to \RR$. We will show that the gluing equations for $\tau$ are in $1$--$1$ correspondence with the leading--trailing deformations; in particular, the non-linear part of a gluing equation is satisfied if and only if the derivative of $\vol$ vanishes along the corresponding deformation. It will follow that the full system of gluing equations is satisfied at $p \in \ang(\tau)$ if and only $p$ is a critical point of $\vol$.

\subsection{Extending and applying the method} Section \ref{sec:extensions}, at the end of the paper, surveys some of the ways in which the Casson--Rivin program has been generalized and applied. Among the generalizations are versions of Theorem \ref{thm:volume-max} for manifolds with polyhedral boundary (Theorem \ref{thm:polyhedra}) and Dehn fillings of cusped manifolds (Theorem \ref{thm:angled-surgery}). Another generalization of the method considers angles modulo $2\pi$: we discuss this briefly while referring to Luo's paper in this volume \cite{luo:survey} for a much more thorough treatment.

Among the applications are explicit constructions of the hyperbolic metric on several families of $3$--manifolds, including punctured torus bundles \cite{gf:bundles}, certain link complements \cite{gueritaud:thesis}, and ``generic'' Dehn fillings of multi-cusped manifolds \cite{gueritaud-schleimer}. For each of these families of manifolds, the volume associated to \emph{any} angle structure gives useful lower bounds on the hyperbolic volume of the manifold: see Theorem \ref{thm:vol-estimate}. For each of these families of manifolds, certain angle inequalities obtained while proving that $\vol: \ang(\tau) \to \RR$ has a critical point also imply that the combinatorially natural triangulation $\tau$ is in fact the \emph{geometrically canonical} way to subdivide $M$. See Section \ref{sec:canonical} for details.

Finally, the behavior of the volume functional $\vol: \ang(\tau) \to \RR$ gives a surprising and short proof of Weil's local rigidity theorem for hyperbolic metrics \cite{weil:rigidity}. Because $\vol$ is strictly concave down (Lemma \ref{lemma:vol-deriv}), any critical point of $\vol$ on $\ang(\tau)$ must be unique. As a result, there is only one complete metric on $M$ in which all the tetrahedra of $\tau$ are positively oriented. This turns out to imply Weil's local rigidity theorem, as we shall show in Theorem \ref{thm:weil-rigidity}.

\subsection{A reader's guide} For a relative novice to this subject, we recommend the following strategy. Skim over Section \ref{sec:gluing}, as needed, to recall the crucial definitions of gluing equations and holonomy. Read carefully the definitions and opening discussion of Sections \ref{sec:angle-struct} and \ref{sec:lead-trail}, but skip the proofs in those sections entirely. Then, proceed directly to Section \ref{sec:volume-max}, where the computation of Proposition \ref{prop:volume-computation} is particularly important. This computation almost immediately implies the forward direction of Theorem \ref{thm:volume-max}. On the other hand, the 
%
%
%
linear--algebra results in Sections \ref{sec:angle-struct} and \ref{sec:lead-trail} involve completely different ideas compared to the rest of the paper, and are only needed for the reverse direction of Theorem \ref{thm:volume-max} (complete metric implies critical point of $\vol$) and for various applications in Section \ref{sec:extensions}.

For a practical example of the method, the reader may also study
the proof of \cite[Lemma 6.2]{gf:bundles}, which proves Theorem \ref{thm:volume-max} in the special case of punctured--torus bundles. 


\subsection{Acknowledgements} The proof of Theorem \ref{thm:volume-max} was developed while the two authors were preparing to give a mini-course at Osaka University in January 2006. This argument was then tested and refined in front of three workshop audiences containing many graduate students: at Osaka University in 2006, at Zhejiang University in 2007, and at Columbia University in 2009. We thank all the participants of those workshops for their suggestions. We are also grateful to the organizers of the Columbia workshop, ``Interactions Between Hyperbolic Geometry, Quantum Topology and Number Theory,'' for their encouragement to write up this material for the workshop proceedings.  Finally, this paper benefited from a number of conversations with Marc Culler, Feng Luo, Igor Rivin, Makoto Sakuma, and Louis Theran.

\section{Gluing equations}\label{sec:gluing}

For the length of this paper, $M$ will typically denote a compact, connected, orientable $3$--manifold whose boundary consists of tori. For example, such a manifold can be obtained from $S^3$ by removing an open tubular neighborhood of a knot or link. 
An \emph{ideal triangulation} $\tau$ is a subdivision of $M \setminus \bdy M$ into ideal tetrahedra, glued in pairs along their faces. One way to recover the compact manifold $M$ is to truncate all the vertices of the tetrahedra: then, the triangles created by the truncation will fit together to tile $\bdy M$.

As discussed in the Introduction, our eventual goal is to find a complete hyperbolic metric on $M$ by gluing together metric tetrahedra. The metric models for the tetrahedra in $M$ come from ideal tetrahedra in $\HH^3$:

\begin{define}\label{def:hyp-ideal}
A \emph{hyperbolic ideal tetrahedron}  $T$ is the convex hull in $\HH^3$ of four distinct points on $\bdy \HH^3$. The four points on $\bdy \HH^3$ are called  \emph{ideal vertices} of $T$, and are not part of the tetrahedron. The tetrahedron $T$ is called \emph{degenerate} if it lies in a single plane, and \emph{non-degenerate} otherwise.
\end{define}

Recall that an isometry of $\HH^3$ is completely determined by its action on three points on $\bdy \HH^3$. Thus we may assume, for concreteness, that three vertices of $T$ lie at $0,1, \infty$ in the upper half-space model. If $T$ is non-degenerate, we may also assume that the fourth vertex lies at $z \in \CC$, with $\Imm (z) > 0$. This number $z \in \CC$ determines $T$ up to isometry.

Suppose we move $T$ by an orientation--preserving isometry of $\HH^3$, so that three vertices again lie at $0,1,\infty$, the fourth vertex again lies in the upper half-plane, and edge $e$ is mapped to the edge $0\infty$. Then the fourth vertex of $T$ will be sent to one of
\begin{equation}\label{eq:shape-param}
z, \quad z' = \frac{z-1}{z}  \quad \mbox{or} \quad z'' = \frac{1}{1-z}.
\end{equation}
The corresponding number $z$, $z'$, or $z''$ is called the \emph{shape parameter} of $e$. Notice that opposite edges of $T$ have the same shape parameter, and that $zz'z''=-1$. The arguments 
$$\arg z > 0, \quad \arg z' > 0, \quad \arg z'' > 0$$ represent the dihedral angles of $T$, or equivalently the angles of the Euclidean triangle with vertices $0,1,z$. If one of the vertices of $T$ is truncated by a horosphere, the intersection will be a Euclidean triangle in precisely this similarity class. This Euclidean triangle of intersection between $T$ and a horosphere is called a \emph{boundary triangle} of $T$. See Figure \ref{fig:shearing}, left.

\begin{figure}
\begin{center}
\psfrag{a}{$\alpha$}
\psfrag{b}{$\gamma$}
\psfrag{g}{$\beta$}
\psfrag{0}{$0$}
\psfrag{1}{$1$}
\psfrag{i}{$\infty$}
\psfrag{z}{$z$}
\includegraphics{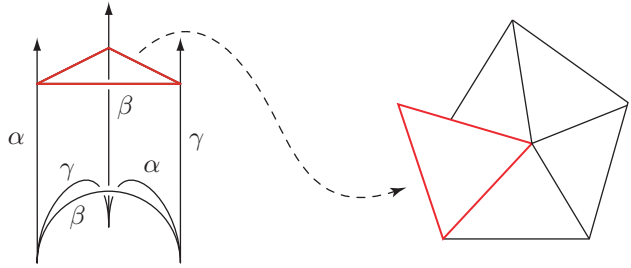}
\caption{Left: the shape of an ideal tetrahedron is defined by three dihedral angles. Right: the gluing of tetrahedra, seen from an ideal vertex.}
\label{fig:shearing}
\end{center}
\end{figure}

The tiling of $\bdy M$ by boundary triangles of the tetrahedra provides a way to understand what conditions are required to glue the tetrahedra coherently.

\begin{define}\label{def:normal}
Let $C$ be an oriented surface with a specified triangulation (for example, a boundary torus of $M$ with the tessellation by  boundary triangles). A \emph{segment} in $C$ is an embedded arc in one triangle, which is disjoint from the vertices of $C$, and whose endpoints lie in distinct edges of $C$. A \emph{normal closed curve} $\sigma \subset C$ is an 
%
%
immersed closed curve that is transverse to the edges of $C$, such that the intersection between $\sigma$ and a triangle is a union of segments.
\end{define}

\begin{define}\label{def:holonomy}
Let $\tau$ be an ideal triangulation of $M$, and let $C$ be one torus component of $\bdy M$. Then $\tau$ induces a tessellation of $C$  by boundary triangles. If each tetrahedron of $\tau$ is assigned a hyperbolic shape, then every corner of each boundary triangle can be labeled with the corresponding shape parameter.

\begin{figure}
\begin{center}
\psfrag{z1}{$z_1$}
\psfrag{z2}{$z_2$}
\psfrag{z3}{$z_3$}
\psfrag{z4}{$z_4$}
\psfrag{z5}{$z_5$}
\psfrag{z6}{$z_6$}
\psfrag{z7}{$z_7$}
\psfrag{z8}{$z_8$}
\includegraphics{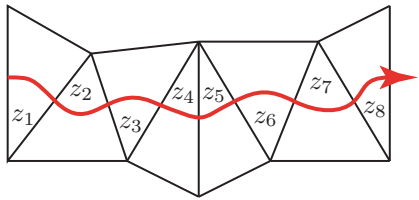}
\caption{A normal path $\sigma$ passing through the tessellation of a torus by boundary triangles.}
\label{fig:holonomy}
\end{center}
\end{figure}

Let $\sigma$ be an oriented normal closed curve in $C$. Then every segment of $\sigma$ in a triangle of $C$  cuts off a single vertex of a triangle, labeled with a single shape parameter (see Figure \ref{fig:holonomy}). Let $z_1, \ldots, z_k$ be the sequence of shape parameters corresponding to the segments of $\sigma$. Then the \emph{holonomy} of $\sigma$ is defined to be
\begin{equation}\label{eq:holonomy}
H(\sigma) = \sum_{i=1}^{k} \epsilon_i \log( z_i),
\end{equation}
where $\epsilon_i = 1$ for corners of triangles to the left of $\sigma$ and $\epsilon_i = -1$ for corners of triangles to the right of $\sigma$. We always choose the branch of the log where $0< \arg(z_i) <\pi$ for $\Imm(z_i)>0$.
\end{define}

\begin{define}\label{def:gluing}
Thurston's \emph{gluing equations} for a triangulation $\tau$ require, in brief, that the holonomy of every curve $\sigma \subset \bdy M$ should be trivial. More precisely, the system consists of \emph{edge equations} and \emph{completeness equations}. The edge equation for an edge $e$ of $\tau$ requires that all the shape parameters adjacent to $e$ satisfy
\begin{equation}\label{eq:edge-gluing}
\sum_{i=1}^{k} \log( z_i) = 2 \pi i.
\end{equation}
The \emph{completeness equations} require that, if $\bdy M$ consists of $k$ tori, there should be a collection of simple closed normal curves $\sigma_1, \ldots, \sigma_{2k}$ spanning $H_1(\bdy M)$, such that
\begin{equation}\label{eq:completeness}
H(\sigma_j) = 0 \quad \forall j.
\end{equation}
\end{define}

As the name suggests, the completeness equations ensure we obtain a complete metric.

\begin{prop}\label{prop:completeness}
Let $\tau$ be an ideal triangulation of $M$. Suppose that each ideal tetrahedron of $\tau$ is assigned a non-degenerate hyperbolic shape with positive imaginary part, as above. If the shape parameters of the tetrahedra satisfy the edge gluing equations, these metric tetrahedra can be glued together to obtain a (possibly incomplete) hyperbolic metric on $M \setminus \bdy M$. This metric will be complete if and only if the completeness equations are satisfied.
\end{prop}

\begin{proof}
Because all hyperbolic ideal triangles are isometric, there is no obstruction to gluing the tetrahedra isometrically along the interiors of their faces. If the edge equation is satisfied for every edge of $\tau$, then this isometric gluing along faces of $\tau$ extends continuously across the edges. Thus we obtain a (possibly incomplete) hyperbolic metric on $M \setminus \bdy M$ and a developing map $D: \widetilde{M} \to \HH^3$.

If the metric on $M$ is complete, then $D: \widetilde{M} \to \HH^3$ is an isometry. It is well-known that the deck transformations corresponding to a component of $\bdy M$ must be parabolic isometries of $\HH^3$. 
In particular, if a cusp of $M$ is lifted to $\infty$ in the upper half-space model of $\HH^3$, the parabolic isometries preserving that cusp are Euclidean translations of $\CC$. But a simple closed curve $\sigma \subset \bdy M$ realizes a translation of $\CC$ precisely when its holonomy is $H(\sigma)=0$. Thus equation (\ref{eq:completeness}) holds for all cusps of $M$.

Conversely, suppose that the completeness equations (\ref{eq:completeness}) are satisfied for a basis of $H_1(\bdy M)$. Then, for every cusp torus $C \subset \bdy M$, the boundary triangles of $C$ can be developed to tile a horosphere in $\HH^3$. In other words, the developing image of a collar neighborhood of $C$ is a horoball in $\HH^3$, which is complete. But if we remove the collar neighborhoods of every boundary torus in $M$, the remaining set is compact, hence also complete. Thus the metric on $M$ is complete.
\end{proof}


\section{The polytope of angle structures}\label{sec:angle-struct}

As we have seen in the last section, the shape of a hyperbolic ideal tetrahedron is completely determined by its dihedral angles. The notion of an angle structure is that these dihedral angles should fit together coherently.

\begin{define}\label{def:angled-polytope}
An \emph{angle structure} on an ideal triangulation $\tau$ is an assignment of an (internal) dihedral angle to each edge of each tetrahedron, such that opposite edges carry the same dihedral angle, and such that
\begin{enumerate}
\item all angles lie in the range $(0, \pi)$,
\item around each ideal vertex of a tetrahedron, the dihedral angles sum to $\pi$,
\item around each edge of $M$, the dihedral angles sum to $2\pi$.
\end{enumerate}
The set of all angle structures on $\tau$ is denoted $\ang(\tau)$.
\end{define}

Conditions $(1)$ and $(2)$ in the definition above are precisely 
%
what is
needed to specify a non-degenerate hyperbolic ideal tetrahedron up to isometry. For concreteness, if the three dihedral angles meeting at a vertex of $T$ are labeled $\alpha, \beta, \gamma$ in clockwise order, then the shape parameter corresponding to $\alpha$ is
\begin{equation}\label{eq:angle-shape}
z(\alpha) =  \frac{\sin \gamma}{\sin \beta}\, e^{i \alpha},
\end{equation}
by the law of sines. Meanwhile, condition $(3)$ in the definition is nothing other than the imaginary part of the edge equation $(\ref{eq:edge-gluing})$. 

The notion of an angle structure can be summarized by saying that the tetrahedra of $\tau$ carry genuine hyperbolic shapes, but the conditions these shapes must satisfy are much weaker than the gluing equations. The completeness equations are discarded entirely\protect\footnote{There is an alternative version of an angle structure, in which the angles must also satisfy the imaginary part of the completeness equations. This corresponds to taking a linear slice of the polytope $\ang(\tau)$. We will not need this version for the main part of the paper --- but see Section \ref{sec:filling} for variations on this theme.}, and the real part of the edge equations is also discarded. If we attempt to glue the metric tetrahedra coming from an angle structure, we can encounter \emph{shearing singularities} at the edges of $\tau$, as in Figure \ref{fig:shearing}.

One may separate the gluing equations of Definition \ref{def:gluing} into a real part and an imaginary part. The real part of the equations is non-linear, because the shape parameters  within one tetrahedron are related in a non-linear way  in equation (\ref{eq:shape-param}). On the other hand, the imaginary part of the equations is linear in the dihedral angles, and 
the set of angle structures $\ang(\tau)$
is defined by a system of linear equations and strict linear inequalities. This leads to

\begin{prop}\label{prop:polytope}
Let $\tau$ be an ideal triangulation of $M$, containing $n$ tetrahedra. The set of all ways to assign a real number to each pair of opposite edges of each tetrahedron is naturally identified with $\RR^{3n}$. Then $\ang(\tau)$ is a convex, finite--sided, bounded polytope in $\RR^{3n}$. If $\ang(\tau) \neq \emptyset$, its dimension is 
$$\dim \ang(\tau) = \abs{\tau} + \abs{\bdy M},$$
where $\abs{\tau} = n$ is the number of tetrahedra and $\abs{\bdy M}$ is the number of boundary tori in $M$.
\end{prop}

\begin{proof}
Condition (1) of Definition \ref{def:angled-polytope} is a system of strict inequalities that constrains the coordinates of $\ang(\tau)$ to the open cube $(0, \pi)^{3n}$. Meanwhile, conditions (2) and (3) impose a system of linear equations, whose solution set is an affine subspace of $\RR^{3n}$. The intersection between this affine subspace and $(0, \pi)^{3n}$ will be a bounded, convex, finite--sided polytope.

We claim that the number of edges in the triangulation $\tau$ is $n$, the same as the tetrahedra. To see this, observe that each tetrahedron gives rise to $4$ boundary triangles that lie in $\bdy M$. Thus $\bdy M$ is subdivided into $4n$ boundary triangles. These triangles have a total of $12n$ sides, glued in pairs; thus $\bdy M$ has $6n$ edges. Since every component of $\bdy M$ is a torus,
$$\chi(\bdy M) = 0 = 4n - 6n + 2n,$$
hence there are $2n$ vertices on $\bdy M$. Since every edge of $\tau$ accounts for two vertices on $\bdy M$, the number of edges in $\tau$ is $n$, as claimed. It will be convenient to think of the tetrahedra as numbered $1$ to $n$, and the edges as numbered $n+1$ to $2n$.

Definition \ref{def:angled-polytope} involves $2n$ equations and $3n$ unknowns. This system can be encoded in a $(2n \times 3n)$ matrix $A$, as follows. Each column of $A$ corresponds to a pair of opposite edges of one tetrahedron. For $1 \leq i \leq n$, the entry $a_{ij}$ records whether or not the $i^{\thh}$ tetrahedron contains the $j^{\thh}$ edge pair. In other words, $a_{ij}$ will be $1$ when $3i -2 \leq j \leq 3i$, and $0$ otherwise. For $n+1 \leq i \leq 2n$, the entry $a_{ij}$ records how many edges out of the $j^{\thh}$ edge pair become identified to the $i^{\thh}$ edge of the glued-up manifold. Thus the entries in the bottom half of $A$ can be $0$, $1$, or $2$. With this setup, the system of equations is
$$A\, \vv = [\pi, \, \ldots, \pi, 2\pi, \ldots, 2\pi]^T,$$
where $\vv \in \RR^{3n}$ is the vector of  angles. Then $\dim \ang(\tau) = 3n - \rank (A)$. 

\begin{lemma}\label{lemma:rank}
$\rank(A) = 2n -  \abs{\bdy M}$.
\end{lemma}

This result is due to Neumann \cite{neumann:triangulations}. Our proof is adapted from  Choi \cite[Theorem 3.7]{choi:triangulation}.

\begin{proof}
Since the matrix $A$ has $2n$ rows, it suffices to show that the row null space has dimension $\abs{\bdy M}$. We will do this by constructing an explicit basis for the row null space, with basis vectors in $1-1$ correspondence with the cusps of $M$.

Let $c$ be a cusp of $M$. Associated to $c$, we construct a row vector $\mathbf{r}_c \in \RR^{2n}$. For $1 \leq i \leq n$, the $i^\thh$ entry of $\rr_c$ is \emph{minus} the number of ideal vertices that the $i^\thh$ tetrahedron has at cusp $c$. For $n+1 \leq i \leq 2n$, the $i^\thh$ entry of $\rr_c$ is \emph{plus} the number of endpoints that the $i^\thh$ edge has at cusp $c$. Thus $\rr_c$ records incidence, with tetrahedra counted negatively and edges counted positively.

\smallskip

\noindent \emph{ \underline{Claim:} For every cusp $c$, $\rr_c \, A = 0$}.

\smallskip

Let $\vv_j$ be a column of $A$, and recall that $\vv_j$ corresponds to one pair of opposite edges in one tetrahedron. Then the dot product $\rr_c \cdot \vv_j$ counts \emph{minus} the number of times that one of these edges has an endpoint at $c$, \emph{plus} the number of times that one of these edges is identified to an edge of $M$ that has an endpoint at $c$. Naturally the sum is $0$.

\smallskip

\noindent \emph{\underline{Claim:} The collection of vectors $\rr_c $ is linearly independent.}

\smallskip

Suppose that $\sum \lambda_c \rr_c = 0$, and let $e_i$ be the $i^\thh$ edge in $M$. The $i^\thh$ entry of $\rr_c$ will be $0$, except when $c$ is one of the endpoints of $e_i$. Thus, if the edge $e_i$ has endpoints at cusps $a$ and $b$, and the $i^\thh$ entry of $\sum \lambda_c \rr_c$ is $0$, we must have $\lambda_a + \lambda_b = 0$.

Now, let $\Delta$ be an ideal triangle of $\tau$, whose ideal vertices are at cusps $a$, $b$, and $c$. The argument above, applied to the three edges of $\Delta$, gives
$$\lambda_a + \lambda_b = 0, \qquad \lambda_a + \lambda_c = 0, \qquad \lambda_b + \lambda_c = 0.$$
The only way that these three equalities can hold simultaneously is if $\lambda_a = \lambda_b = \lambda_c = 0$. Thus all coefficients $\lambda_c$ must vanish.
\smallskip

\noindent \emph{\underline{Claim:} The collection of vectors $\rr_c $ spans the row null space of $A$.}

\smallskip

Let $\qq = [q_1, \ldots, q_{2n}]$ be a vector such that $\qq \, A = 0$. Our goal is to show that $\qq =  \sum \lambda_k \rr_k$.
To that end, let $T_i$, $(1 \leq i \leq n)$, be one tetrahedron of $\tau$. Let $a,b,c,d$ be the cusps of $M$ corresponding to the four ideal vertices of $T_i$. (Some of these cusps may coincide.) A pair of letters from the collection $\{a,b,c,d\}$ determines an edge of $T_i$, which is identified to some edge of $M$. For notational convenience, suppose that edge $ab$ from tetrahedron $T_i$ is identified to the edge $i(ab)$ in $M$, and similarly for the other letters.

Now, let $\vv_j$ be the $j^\thh$ column vector of $A$. This column of $A$ corresponds to a pair of opposite edges in one tetrahedron, which will be tetrahedron $T_i$ if and only if $3i -2 \leq j \leq 3i$. After relabeling $\{a,b,c,d\}$, we may assume that $j = 3i -2$ corresponds to opposite edges $i(ab)$ and $i(cd)$, \quad  $j = 3i -1$ corresponds to $i(ac)$ and $i(bd)$, \quad and $j = 3i$ corresponds to $i(ad)$ and $i(bc)$. By the definition of the matrix $A$, the only non-zero entries of $\vv_{3i -2}$ are in row $i$ (corresponding to the tetrahedron $T_i$) and rows $i(ab), \:  i(cd)$ (corresponding to edges of $M$). The analogous statement holds for $j = 3i -1$ and $j = 3i$.

Since $\qq \, A = 0$, we must have $\qq \cdot \vv_j = 0$ for every $j$. Applying this to $3i -2 \leq j \leq 3i$ gives
$$
\begin{array}{r c c c l}
\qq \cdot \vv_{3i-2} &=& q_i + q_{i(ab)} + q_{i(cd)} &=& 0, \\ 
\qq \cdot \vv_{3i-1} &=& q_i + q_{i(ac)} + q_{i(bd)} &=& 0, \\
\qq \cdot \vv_{3i} &=& q_i + q_{i(ad)} + q_{i(bc)} &=& 0, 
\end{array}
$$
which implies
\begin{equation}\label{eq:q-depend}
q_{i(ab)} + q_{i(cd)} \: \: = \: \:
 q_{i(ac)} + q_{i(bd)} \: \: = \: \:
q_{i(ad)} + q_{i(bc)} \: \: = \: \:
- q_i.
\end{equation}

We are now ready to find coefficients $\lambda_k$ such that $\qq =  \sum \lambda_k \rr_k$. Let $\Delta$ be an ideal triangle in $T_i$, whose ideal vertices are at cusps $a,b,c$ and whose sides are at edges $i(ab), \:  i(ac), \:  i(bc)$. Define
\begin{equation}\label{eq:lambda-def}
\lambda_a = \frac{ q_{i(ab)} \! + \!  q_{i(ac)} \! - \! q_{i(bc)} }{2}, \quad
\lambda_b = \frac{ q_{i(ab)} \! + \!  q_{i(bc)} \! - \!  q_{i(ac)} }{2}, \quad
\lambda_c = \frac{ q_{i(ac)} \! + \!  q_{i(bc)} \! - \!  q_{i(ab)} }{2}.
\end{equation}
A priori, this definition depends on the ideal triangle $\Delta$. To check this is well-defined, let $\Delta'$ be another ideal triangle of $T_i$, for example the triangle with ideal vertices at $b,c,d$. Note that equation (\ref{eq:q-depend}) implies that
$$q_{i(cd)} - q_{i(bd)} = q_{i(ac)} - q_{i(ab)}.$$
Thus the definition of $\lambda_c$ coming from $\Delta'$ will be
$$\lambda_c \: = \: \frac{ q_{i(bc)} + q_{i(cd)} - q_{i(bd)} }{2} \: = \: \frac{ q_{i(bc)} + q_{i(ac)} - q_{i(ab)} }{2}  .$$
Therefore, triangles $\Delta$ and $\Delta'$ that belong to the same tetrahedron and share a vertex at cusp $c$ induce the same definition of $\lambda_c$. Since any pair of ideal triangles that meet the cusp $c$ are connected by a sequence of tetrahedra, it follows that $\lambda_c$ is well-defined.

It remains to check that this definition of $\lambda_c$ satisfies $\qq =  \sum \lambda_k \rr_k$. Recall that for $n+1 \leq i \leq 2n$, the 
$i^\thh$ entry of $\rr_k$ is the number of endpoints that the $i^\thh$ edge has at cusp $k$. If the ends of this edge are at cusps $a$ and $b$, i.e.\ if this is the edge we have referred to as $i(ab)$, then the $i^\thh$ entry of $\sum \lambda_k \rr_k$ is
\begin{equation}\label{eq:lambda-sum}
\lambda_a + \lambda_b  = q_{i(ab)},
\end{equation}
using equation (\ref{eq:lambda-def}). This is exactly the $i^\thh$ entry of $\qq$. Similarly, for $1 \leq i \leq n$, the $i^\thh$ entry of $\rr_k$ is \emph{minus} the number of ideal vertices that $T_i$ has at cusp $k$. Thus, if $T_i$ has ideal vertices at cusps  $a,b,c,d$ (where some of these may coincide), the $i^\thh$ entry of $\sum \lambda_k \rr_k$ is
$$
-(\lambda_a + \lambda_b + \lambda_c + \lambda_d) \: = \: - q_{i(ab)} - q_{i(cd)} \: = \: q_i,$$
using equations (\ref{eq:lambda-sum}) and (\ref{eq:q-depend}).  Thus $\qq =  \sum \lambda_k \rr_k$, as desired. 

This completes the proof that the vectors $\rr_c$, indexed by the cusps of $M$, form a basis for the row null space of $A$. Therefore, $\rank(A) = 2n -  \abs{\bdy M}$
\end{proof}

We conclude that $\dim \ang(\tau) = 3n - \rank (A) =  n + \abs{\bdy M}$, completing the proof of Proposition \ref{prop:polytope}.
\end{proof}

\section{Leading--trailing deformations}\label{sec:lead-trail}

Given a point $p$ of the angle polytope $\ang(\tau)$, let  $T_p \ang(\tau)$ be  the tangent space to $\ang(\tau)$ at the point $p$. There is a particularly convenient choice of spanning vectors for $T_p \ang(\tau)$. These \emph{leading--trailing deformations}, which were probably folklore knowledge to experts for several years, are the main innovation in our direct proof of Theorem \ref{thm:volume-max}. In a certain sense, they serve as a more concrete reformulation of Neumann and Zagier's symplectic pairing \cite{neumann-zagier}. To the best of the authors' knowledge, this is the first place where they are described in print.

\begin{define}\label{def:lead-trail}
Let $C$ be a cusp torus of $M$, with a tessellation by boundary triangles coming from $\tau$. Let $\sigma \subset C$ be an oriented normal closed curve (see Definition \ref{def:normal}), consisting of segments $\sigma_1, \ldots \sigma_k$. Every oriented segment $\sigma_i$ lies in a boundary triangle $\Delta_i$. We define  the \emph{leading} corner of $\Delta_i$ to be the corner opposite the side where $\sigma_i$ enters $\Delta_i$, and mark it with a $+$. We define  the \emph{trailing} corner of $\Delta_i$ to be the corner opposite the side where $\sigma_i$ leaves $\Delta_i$, and mark it with a $\ominus$. See Figure \ref{fig:lead-trail}.

\begin{figure}[h]
\begin{center}
\psfrag{p}{$+$}
\psfrag{m}{$\ominus$}
\includegraphics{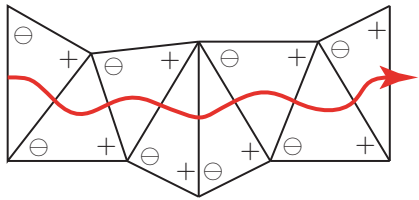}
\caption{The leading--trailing deformation along $\sigma$ increases the angles marked $+$ and decreases the angles marked $\ominus$.}
\label{fig:lead-trail}
\end{center}
\end{figure}

For every oriented segment $\sigma_i$, we define a vector $\ww(\sigma_i) \in \RR^{3n}$, where as above each coordinate of $\RR^{3n}$  corresponds to one pair of opposite edges in one of the $n$ tetrahedra. Every corner of $\Delta_i$ corresponds to one such edge pair in a tetrahedron. The vector $\ww(\sigma_i)$ will have a $1$ in the coordinate corresponding to the leading corner of $\Delta_i$, a $-1$ in the coordinate corresponding to the trailing corner of $\Delta_i$, and $0$'s otherwise.

Finally, the \emph{leading-trailing deformation} corresponding to $\sigma$ is the vector $\ww(\sigma) = \sum_i \ww(\sigma_i)$.
\end{define}

\begin{example}\label{ex:leading-edge}
Let $e$ be an edge of $M$, and suppose for simplicity that all tetrahedra adjacent to $e$ are distinct. Let $\sigma$ be a simple closed curve on $\bdy M$, running counterclockwise about one endpoint of $e$. The boundary triangles intersected by $\sigma$ have angles $\alpha_i, \beta_i, \gamma_i$, labeled clockwise with $\alpha_i$ inside $\sigma$. Then the leading--trailing deformation $\ww(\sigma)$ adds $1$ to every $\beta_i$ and subtracts $1$ from every $\gamma_i$, keeping the angle sum in each tetrahedron equal to $\pi$. In addition, the dihedral angle marked $\beta_i$ is adjacent to the same edge of $M$ as the dihedral angle marked $\gamma_{i+1}$. Thus the angle sum at each edge is unchanged. (See Figure \ref{fig:tet-neighborhood}.) Observe that $\ww(\sigma)$ has no effect at all on the dihedral angles $\alpha_i$ adjacent to $e$.
\end{example}

\begin{figure}[h]
\begin{center}
\psfrag{a}{$\alpha_1$}
\psfrag{b}{$\alpha_2$}
\psfrag{c}{$\alpha_3$}
\psfrag{d}{$\alpha_4$}
\psfrag{e}{$\alpha_5$}
\psfrag{f}{$e$}
\psfrag{p}{$+$}
\psfrag{m}{$\ominus$}
\includegraphics[width=5in]{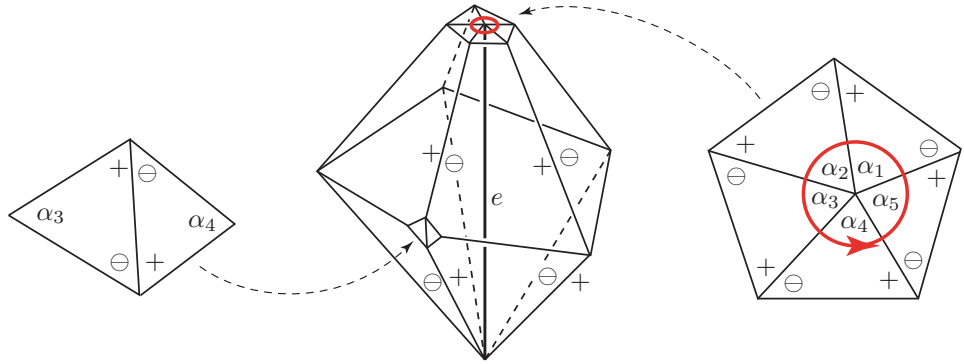}
\caption{The leading--trailing deformation about a single edge $e$ of $M$. Every increase to an angle in a tetrahedron is canceled out by a decrease to an adjacent angle.}
\label{fig:tet-neighborhood}
\end{center}
\end{figure}

In the more general setting, we will prove in Lemma \ref{lemma:tangent} that leading--trailing deformations are always tangent to $\ang(\tau)$. To do that, we need a better understanding of the interaction between different deformations.


\begin{define}\label{def:int-num}
Let $\tau$ be an ideal triangulation of $M$, and let $\rho, \sigma$ be oriented normal closed curves on $\bdy M$ that intersect transversely (if at all). Define the \emph{signed intersection number} $\iota(\rho, \sigma)$ to be the number of times that $\sigma$ crosses $\rho$ from right to left, minus the number of times that $\sigma$ crosses $\rho$ from left to right. This definition has a few immediate properties:
\begin{itemize}
\item It is anti-symmetric: $\iota(\rho, \sigma) = -\iota(\sigma, \rho)$.
\item It depends only on the homology classes of $\sigma$ and $\rho$ in $H_1(\bdy M)$.
\item By considering a transverse pushoff of $\sigma$, one obtains $\iota(\sigma, \sigma) = 0$.
\end{itemize}
\end{define}


\begin{lemma}\label{lemma:deformation-crossing}
Let $\rho, \sigma$ be oriented normal closed curves on $\bdy M$ that intersect transversely, if at all. Then
$$\frac{\partial}{\partial \ww(\sigma)} \Imm(H(\rho)) = 2 \, \iota(\rho, \sigma).$$
\end{lemma}

Recall that  $\Imm(H(\rho))$ is the linear, angled part of the holonomy in Definition \ref{def:holonomy}.

\begin{proof} The proof involves three steps. \smallskip

\emph{\underline{Step 1}} introduces  several simplifying assumptions with no loss of generality. 
First, it will help to assume that every tetrahedron $T$ is embedded in $M$ (that is, $T$ does not meet itself along an edge or face). This assumption can always be met by passing to a finite--sheeted cover of $M$, since $\pi_1(M)$ is residually finite by Selberg's lemma. Note that the tetrahedra, angles, boundary curves, holonomies, and leading--trailing deformations all lift naturally to covers. This assumption also implies that boundary triangles $\Delta_i, \Delta_j$ meet in either at most one edge or at most one vertex.

\begin{figure}
\begin{center}
\psfrag{s}{$\sigma$}
\psfrag{sp}{$\sigma'$}
\psfrag{sd}{$\sigma''$}
\psfrag{R}{$\Rightarrow$}
\includegraphics[width=5in]{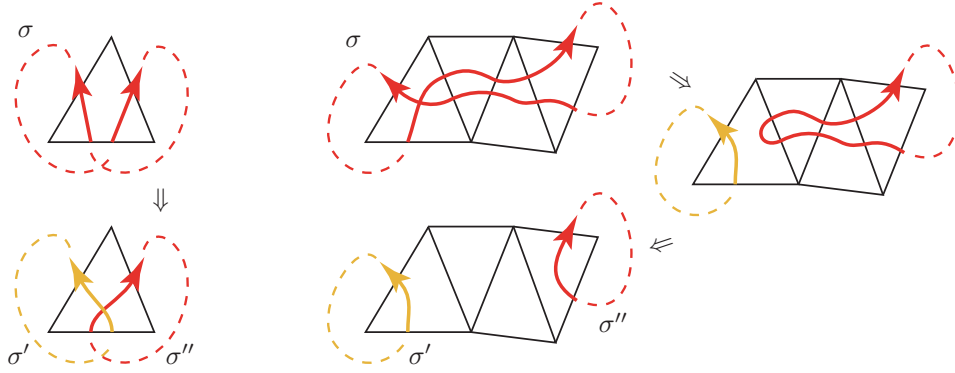}
\caption{Decomposing $\sigma$ into smaller curves. Left: when two segments of $\sigma$ pass through the same boundary triangle, we may cut and rejoin to form two normal curves $\sigma'$, $\sigma''$. Right: a local isotopy in the complement of the vertices may be required to ensure that $\sigma'$, $\sigma''$ are again normal curves.}
\label{fig:cut-join}
\end{center}
\end{figure}

Next, no generality is lost by assuming that $\sigma$ 
passes through each boundary triangle at most once. For,   if two segments $\sigma_i, \sigma_j$ run through the same boundary triangle $\Delta$, we may cut and rejoin the curve $\sigma$ into a pair of shorter normal closed curves $\sigma', \sigma''$. See Figure \ref{fig:cut-join}. (To ensure that $\sigma'$ and $\sigma''$ are \emph{normal} curves, it may be necessary to move them by isotopy in the complement of the vertices, as in the right panel of the figure.) Note that this operation is natural and topologically additive: we have $\ww(\sigma) = \ww(\sigma') + \ww(\sigma'')$, and $\iota(\rho, \sigma) = \iota(\rho, \sigma') + \iota(\rho, \sigma'')$. Thus, if we prove the lemma for each of $\sigma'$ and $\sigma''$, the result will follow for $\sigma$. In particular, we may now assume that $\sigma$ is embedded.

Similarly, we may assume that $\rho$ also passes through each boundary triangle at most once. For, if $\rho_i, \rho_j$ run through the same boundary triangle $\Delta$, we may decompose $\rho$ into a pair of curves $\rho'$, $\rho''$ just as above. This natural operation ensures $H(\rho) = H(\rho') + H(\rho'')$, and $\iota(\rho, \sigma) = \iota(\rho', \sigma) + \iota(\rho'', \sigma)$. Thus it suffices to consider each of $\rho'$ and $\rho''$ in place of $\rho$.



Given these simplifying assumptions, suppose that the segments $\sigma_1, \ldots, \sigma_k$ of $\sigma$ are contained in distinct boundary triangles $\Delta_1, \ldots, \Delta_k$. A priori, there are two ways in which $\ww(\sigma)$ could affect the angular component of $H(\rho)$:
\begin{enumerate}
\item A segment of $\rho$ lies in the same boundary triangle $\Delta_i$ that contains $\sigma_i$. Then, $\ww(\sigma_i)$ changes the angles of $\Delta_i$, thereby changing $\Imm(H(\rho))$.

\item A segment of $\rho$ lies in a boundary triangle $\Delta'_i$, where $\Delta_i \neq \Delta'_i$ are two truncated vertices of the same ideal tetrahedron. By changing dihedral angles in the ambient tetrahedron, $\ww(\sigma_i)$ also changes the angles of $\Delta'_i$.
\end{enumerate}
We consider the effect of (2) in Step 2 and the effect of (1) in Step 3.

\smallskip

\emph{\underline{Step 2}} is to show that the effect of $\ww(\sigma)$ in the boundary triangles ``on the other side'' of the ambient tetrahedra, as in (2) above, will never change $\Imm(H(\rho))$. To see why this is true, it helps to group together boundary triangles that share a common vertex $v$.

Let $v$ be a vertex of $\bdy M$, and let $\sigma_v = \sigma_1 \cup  \ldots \cup \sigma_j$ be a maximal union of consecutive segments in $\sigma$, such that the ambient boundary triangles $\Delta_1 \ldots, \Delta_j$ are all adjacent to $v$.  By our embeddedness assumption, each $\Delta_i$ is adjacent to $v$ in only one corner. Recall that $v$ is one endpoint of an edge $e \subset M$; we will investigate the effect of $\ww(\sigma_v)$ on boundary triangles $\Delta'_1, \ldots, \Delta'_j$ at the other end of edge $e$.  The arc  $\sigma_v$ can take one of three forms:
\begin{itemize}
\item[(i)] $\sigma_v = \sigma$, and forms a closed loop about $v$.
\item[(ii)] $\sigma_v$ takes a right turn in $\Delta_1$, followed by a sequence of left turns about $v$, followed by a right turn in $\Delta_j$.
\item[(iii)] $\sigma_v$ takes a left turn in $\Delta_1$, followed by a sequence of right turns about $v$, followed by a left turn in $\Delta_j$.
\end{itemize}

Scenario (i) is depicted in the right panel of Figure \ref{fig:tet-neighborhood}. Here, $v$ is in the interior of a polygon $P_v  = \Delta_1 \cup \ldots \cup \Delta_j$. Notice that none of the external angles along $\bdy P_v$ actually change, even though the shapes of constituent triangles are changing. The same will be true in the boundary polygon $P'_{v}  = \Delta'_1 \cup \ldots \cup \Delta'_j$ at the other end of $e$. Thus, if $\rho$ passes through $P'_{v}$, its holonomy will be unaffected by this change.

\begin{figure}
\begin{center}
\psfrag{a}{$\alpha_1$}
\psfrag{b}{$\alpha_2$}
\psfrag{c}{$\alpha_3$}
\psfrag{d}{$\alpha_4$}
\psfrag{A}{$$}
\psfrag{B}{$$}
\psfrag{C}{$$}
\psfrag{D}{$$}
\psfrag{p}{$+$}
\psfrag{m}{$\ominus$}
\psfrag{v}{$v$}
\psfrag{Pv}{$P_v$}
\psfrag{Pp}{$P'_v$}
\includegraphics[width=5in]{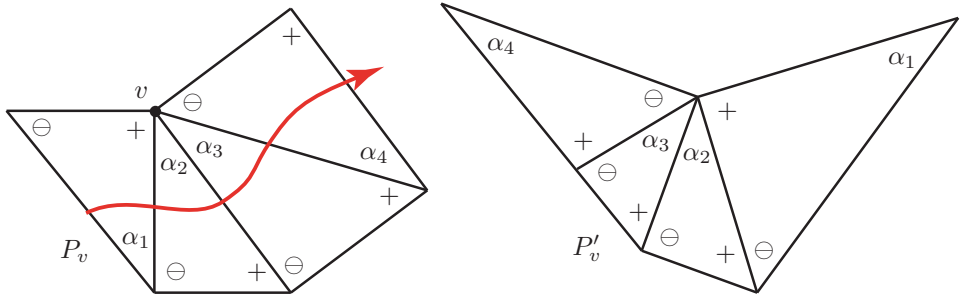}
\caption{Left: the boundary polygon $P_v$ comprised of boundary triangles $\Delta_1, \ldots, \Delta_j$ adjacent to $v$. In each $\Delta_i$, the smallest angle is $\alpha_i$. Right: in the boundary polygon $P'_{v}  = \Delta'_1 \cup \ldots \cup \Delta'_j$, at the other end of the edge that starts at $v$, none of the total angles along $\bdy P'_v$ change.}
\label{fig:both-sides}
\end{center}
\end{figure}

In Scenario (ii), we once again construct a polygon $P_v$ by gluing $\Delta_1$ to $\Delta_2$ along their (unique) shared edge, and so on up to $\Delta_j$.  Let $P'_{v}  = \Delta'_1 \cup \ldots \cup \Delta'_j$ be the polygon formed by the boundary triangles in the same tetrahedra, on the other end of $e$. These two polygons are depicted in Figure \ref{fig:both-sides}. Notice that although $\Delta_i$ and $\Delta'_i$ are in the same oriented similarity class (both have the same angles $\alpha_i, \beta_i, \gamma_i$, in clockwise order), the triangles in these similarity classes are rearranged to form $P'_v$. In particular, $\ww(\sigma_v)$ does not change any of the external angles along $\bdy P'_v$, even though the shapes of constituent triangles $\Delta'_i$ are changing. Thus, if $\rho$ passes through $P'_{v}$, its holonomy will be unaffected.

Scenario (iii) is the mirror image of (ii). Once again, $\ww(\sigma_v)$ does not change any of the external angles along $\bdy P'_v$.

We conclude that for every vertex $v$ in a boundary triangle visited by $\sigma$, the polygon $P'_{v}$ on the other end of the same edge will have all its external angles unaffected by $\ww(\sigma_v)$. Since each segment $\sigma_i$ of $\sigma$ belongs to three different maximal arcs $\sigma_v$ (corresponding to the three distinct vertices of $\Delta_i$), we have $\sum_v \ww(\sigma_v) = 3\ww(\sigma)$. Thus, since each $\ww(\sigma_v)$ does not affect the holonomy of $\rho$ ``on the other side'' of the ambient tetrahedra, neither does $\ww(\sigma)$. This completes Step 2.

\smallskip
 
\emph{\underline{Step 3}}  considers the holonomy of $\rho$ in the same boundary triangles that are also visited by $\sigma$. Consider a maximal consecutive string of segments $\sigma_1, \ldots, \sigma_j$, contained in $\Delta_1, \ldots, \Delta_j$, such that consecutive segments $\rho_1, \ldots, \rho_j$ run through the same  boundary triangles. Depending on orientations, we can have either $\rho_1 \subset \Delta_1$ or $\rho_1 \subset \Delta_j$.

In the special case where $\Delta_1 \cup \ldots \cup \Delta_j$ contains all of $\sigma$ \emph{and} all of $\rho$, the two curves cut off exactly the same corners of boundary triangles (possibly in opposite cyclic order). Then, since $\bdy M$ is orientable, we must have $\iota(\rho, \sigma)=0$. In this case, the angles that go into computing the angular holonomy $\Imm(H(\rho))$ are exactly the angles that are unaffected by $\ww(\sigma)$. Thus, in this special case, $ \partial \Imm(H(\rho)) / \partial \ww(\sigma) = 0 = 2 \, \iota(\rho, \sigma)$.

In the general case, we may assume that $\sigma$ and $\rho$ do not run in parallel through $
\Delta_1$ or through $\Delta_j$ (otherwise, the sequence $1, \ldots, j$ is not maximal). As above, it helps to construct a polygon $P$ by gluing $\Delta_1$ to $\Delta_2$ along their (unique) shared edge, and so on up to $\Delta_j$. By Definition \ref{def:lead-trail}, each $\Delta_i$ has one mark of $+$ in the leading corner relative to $\sigma_i$, and one mark of $\ominus$ in the trailing corner. Altogether, the polygon $P$ contains $2j$ markers, with $j$ pluses and $j$ minuses. We consider how many of these markers are to the left and right of $\rho$.

\begin{figure}
\begin{center}
\psfrag{s}{$\sigma$}
\psfrag{r}{$\rho$}
\psfrag{p}{$+$}
\psfrag{m}{$\ominus$}
\psfrag{R}{$\Rightarrow$}
\includegraphics{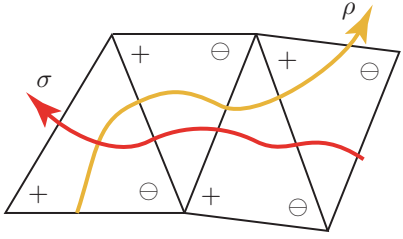}
\caption{When $\sigma$ crosses $\rho$ from right to left, the part of the polygon $P$ to the left of $\rho$ contains an excess of two $+$'s.}
\label{fig:contributions}
\end{center}
\end{figure}

If $\sigma$ enters polygon $P$ to the right of $\rho$ and leaves to the left of $\rho$, the interior of $P$ contributes $+1$ to $\iota(\rho, \sigma)$. Also, the part of $P$ to the left of $\rho$ will have a surplus of two $+$'s. Thus $\ww(\sigma_1 \cup \ldots \cup \sigma_j)$ increases the angular component of $H(\rho)$ by $2$. See Figure \ref{fig:contributions}.

If $\sigma$ enters polygon $P$ to the left of $\rho$ and leaves to the right of $\rho$, the interior of $P$ contributes $-1$ to $\iota(\rho, \sigma)$. Also, the part of $P$ to the left of $\rho$ will have a surplus of two $\ominus$'s. Thus $\ww(\sigma_1 \cup \ldots \cup \sigma_j)$ increases the angular component of $H(\rho)$ by $-2$.

Finally, if $\sigma$ enters and leaves $P$ on the same side of $\rho$, the interior of $P$ contributes $0$ to $\iota(\rho, \sigma)$. Also, the part of $P$ to the left of $\rho$ will have the same number of $+$'s and $\ominus$'s, hence $\ww(\sigma_1 \cup \ldots \cup \sigma_j)$ does not affect the angular component of $H(\rho)$.

Summing these contributions over the (disjoint) polygons of intersection, we conclude that $ \partial \Imm(H(\rho)) / \partial \ww(\sigma) = 2 \, \iota(\rho, \sigma)$, as desired.
 \end{proof}

We can now show that every leading--trailing deformation $\ww(\sigma)$ is tangent to $\ang(\tau)$.

\begin{lemma}\label{lemma:tangent}
Let $p \in \ang(\tau)$ be an angle structure, and let $\sigma$ be an oriented normal curve on a cusp of $M$. Then the vector $\ww(\sigma)$ is tangent to $\ang(\tau)$. In other words, for all sufficiently small $\eps > 0$, \quad  $p + \eps \ww(\sigma) \in \ang(\tau)$.
\end{lemma}

\begin{proof}
Let us check the conditions of Definition \ref{def:angled-polytope}. By condition (1) of the definition, $p$ lies in the open set $(0, \pi)^{3n}$. Thus, for sufficiently small $\eps$, we have $p + \eps \ww(\sigma) \in (0, \pi)^{3n}$ also.

Next, observe that in Definition \ref{def:lead-trail}, every vector $\ww(\sigma_i)$ only affects a single tetrahedron. In this tetrahedron,  $\eps \ww(\sigma_i)$ adds $\eps$ to the dihedral angle on one (leading) pair of opposite edges, and adds $-\eps$ to the dihedral angle on another (trailing) pair of opposite edges. Thus the deformation of angles coming from $\ww(\sigma_i)$ keeps the angle sum equal to $\pi$ in every tetrahedron. Clearly, the same is true for $\ww(\sigma) = \sum_i \ww(\sigma_i)$.

Finally, we check condition (3) of Definition \ref{def:angled-polytope}. Let $e$ be an edge of $M$, and let $\rho$ be a normal closed curve on $\bdy M$ that encircles one endpoint of $e$. Because $\rho$ is homotopically trivial, we have  $\iota(\rho, \sigma) = 0$. Thus, by Lemma \ref{lemma:deformation-crossing}, deforming the dihedral angles along $\ww(\sigma)$ does not change the imaginary part of $H(\rho)$. But $\Imm(H(\rho))$ is nothing other than the sum of dihedral angles about edge $e$; this sum stays constant, equal to $2\pi$.
\end{proof}

There is a convenient choice of leading--trailing deformations that will span $T_p \ang(\tau)$. 

\begin{prop}\label{prop:tangent-span}
For every edge $e_i$ of $M$, where $1 \leq i \leq n$, choose a normal closed curve $\rho_i$ about one endpoint 
of $e_i$. In addition, if $M$ has $k$ cusps, choose simple closed normal curves $\sigma_1, \ldots, \sigma_{2k}$ that will span $H_1(\bdy M)$. Then the vectors $\ww(\rho_i)$ and $\ww(\sigma_j)$ span the tangent space $T_p \ang(\tau)$.
\end{prop}

 Notice that the curves $\rho_1, \ldots, \rho_{n}, \sigma_1, \ldots, \sigma_{2k}$ are exactly the ones whose holonomy is being considered in Definition \ref{def:gluing}.

\begin{proof}
By Proposition \ref{prop:polytope}, $\dim T_p \ang(\tau) = n + k$. Thus it suffices to show that $\ww(\rho_i), \ww(\sigma_j)$ span a vector space of this dimension.

First, we claim that the vectors $\ww(\sigma_1), \ldots, \ww(\sigma_{2k})$ are linearly independent from one another, and from each of the  $\ww(\rho_i)$. Suppose, after renumbering, that $\sigma_1$ and $\sigma_2$ are homology basis curves on the same torus of $\bdy M$. Then $\iota(\sigma_1, \sigma_2) = \pm 1$, while $\iota(\sigma_j, \sigma_2) = 0$ for every $j \neq 1$ (including $j=2$). Similarly, since each $\rho_i$ is homotopically trivial, $\iota(\rho_i, \sigma_2) = 0$ for every $i$. Thus, by Lemma \ref{lemma:deformation-crossing}, $\ww(\sigma_1)$ is the only deformation among the $\rho_i$ and $\sigma_j$ that affects $\Imm(H(\sigma_2))$. Therefore, $\ww(\sigma_1)$ is independent from all the other vectors in the collection. Similarly, each $\ww(\sigma_j)$ is independent from all the other vectors in the collection.

To complete the proof, it will suffice to show that $\ww(\rho_1), \ldots, \ww(\rho_n)$ span a vector space of dimension $n-k$. Let $B$ be the $(n \times 3n)$ matrix whose $i^\thh$ row is $\ww(\rho_i)$. Then, we claim that $\rank(B) = n-k$. The proof of this claim is virtually identical to the proof of Lemma \ref{lemma:rank}. For every cusp $c$, we define a row vector $\rr_c \in \RR^n$, whose $i^\thh$ entry is the number of endpoints that the $i^\thh$ edge $e_i$ has at cusp $c$. (This is the second half of the vector $\rr_c$ from Lemma \ref{lemma:rank}.) Then, by the same argument as in that lemma, one checks that the vectors $\rr_c$ form a basis for the row null space of $B$. Therefore, $\rank(B) = n-k$, and the  $\ww(\rho_i), \ww(\sigma_j)$ span a vector space of dimension $(n-k)+2k$, as required.
\end{proof}

\section{Volume maximization}\label{sec:volume-max}

In this section, we show how volume considerations give a way to turn an angle structure into a genuine hyperbolic metric on $M$.  To compute the volume of an ideal hyperbolic tetrahedron, recall the Lobachevsky function $\lob: \RR \to \RR$. Its definition is
\begin{equation}\label{eq:lobachevsky}
\lob(x) = - \int_0^x \log \abs{2\sin t} \, dt.
\end{equation}

\begin{lemma}\label{lemma:tet-volume}
The Lobachevsky function $\lob(x)$ is well defined and continuous on $\RR$ (even though the defining integral is improper), and periodic with period $\pi$. Furthermore, if $T$ is a hyperbolic ideal tetrahedron with dihedral angles $\alpha, \beta, \gamma$, its volume satisfies
$$\mathrm{vol}(T) = \lob(\alpha) + \lob(\beta) + \lob(\gamma).$$
\end{lemma}

\begin{proof}
See, for example, Milnor \cite{milnor:150years}.
\end{proof}

Following Lemma \ref{lemma:tet-volume}, we may define the volume of an angle structure in a natural way.

\begin{define}\label{def:vol-functional}
Let $\tau$ be an ideal triangulation of $M$, containing $n$ tetrahedra. Let
 $\ang(\tau) \subset \RR^{3n}$ be the polytope of angle structures on $M$. Then we define a \emph{volume functional} $\vol: \overline{\ang(\tau)} \to \RR$, by assigning to a point $p = (p_1, \ldots, p_{3n})$ the real number
$$\vol(p) = \lob(p_1) + \ldots + \lob(p_{3n}).$$
By Lemma \ref{lemma:tet-volume}, $\vol(p)$ is equal to the sum of the volumes of the hyperbolic tetrahedra associated to the angle structure $p$.
\end{define}

\begin{lemma}\label{lemma:vol-deriv}
Let $p = (p_1, \ldots, p_{3n}) \in \ang(\tau)$ be an angle structure on $\tau$, and let \linebreak $\ww = (w_1, \ldots, w_{3n}) \in T_p \ang(\tau)$ be a nonzero tangent vector at $p$. Then the first two derivatives of $\vol(p)$ satisfy
$$\dfrac{\partial \vol}{\partial \ww} = \sum_{i=1}^{3n} -w_i \log \sin p_i
\qquad \mbox{and} \qquad
 \dfrac{\partial^2 \vol}{\partial \ww^2} < 0.$$
In particular, $\vol$ is strictly concave down on $\ang(\tau)$.
\end{lemma}

\begin{proof}
Since the definition of $\vol$ is linear over the tetrahedra in $\tau$, it suffices to consider the volume of one tetrahedron. Thus, suppose that a tetrahedron $T$ has angles $p_1, p_2, p_3 > 0$, which are changing at rates $w_1, w_2, w_3$. Note that, since the tangent vector $\ww$ must preserve the angle sum in each tetrahedron, we have $w_1 + w_2 + w_3 = 0$. We may also assume that at least one (hence, at least two) of the $w_i$ are nonzero. Then, by equation (\ref{eq:lobachevsky}) and Lemma \ref{lemma:tet-volume},
$$\dfrac{\partial \mathrm{vol}(T)}{\partial \ww} \: = \:  \sum_{i=1}^{3}  -  w_i  \log \abs{2\sin p_i} \: = \:  \sum_{i=1}^{3}  -  w_i  \log \sin p_i, $$
because all sines are positive and $\sum w_i \log 2 = 0$. This completes the computation of the first derivative.

To compute the second derivative, assume by symmetry that $p_1, p_2 < \pi/2$. Differentiating $\mathrm{vol}(T)$ a second time, we get
\begin{eqnarray*} 
 - \: \dfrac{\partial^2  \mathrm{vol}(T)}{\partial \ww^2}
&=& w_1^2 \cot p_1 +w_2^2 \cot p_2 + w_3^2 \cot p_3 \\
&=&  w_1^2 \cot p_1 +w_2^2 \cot p_2  +(w_1 + w_2 )^2\frac{1-\cot p_1 \cot p_2}{\cot p_1+ \cot p_2}\\ 
&=&\frac{(w_1 + w_2 )^2+(w_1 \cot p_1 - w_2 \cot p_2)^2}{\cot p_1 + \cot p_2} \\
&\geq& 0.
\end{eqnarray*} 
In fact, the numerator in the next-to-last line must be strictly positive. For the numerator to be $0$, we must have $w_1 = - w_2$, hence $  \cot p_1 = - \cot p_2$, which is impossible when $p_1, p_2 \in (0, \pi/2)$. Thus $\partial^2  \mathrm{vol}(T) / \partial \ww^2 < 0$, and the volume functional $\vol$ is also strictly concave down.
\end{proof}

By Lemma \ref{lemma:vol-deriv}, the only potential critical point of $\vol$ is a global maximum. As the next proposition shows, derivatives of $\vol$ are closely connected to the holonomy of curves on $\bdy M$.

\begin{prop}\label{prop:volume-computation}
Let $C$ be a cusp torus of $M$, with a tessellation by boundary triangles coming from $\tau$. Let $\sigma \subset C$ be an oriented normal closed curve. Recall the holonomy $H(\sigma)$ from Definition \ref{def:holonomy} and the tangent vector $\ww(\sigma) \in T_p \ang(\tau)$ from Definition \ref{def:lead-trail}.

Then for every point $p \in \ang(\tau)$, we have
$$\frac{\partial \vol}{\partial \ww(\sigma)} = Re (H(\sigma)).$$
\end{prop}

\begin{proof}
Let $\sigma_1, \ldots, \sigma_k$ be the segments of $\sigma$, with $\sigma_i$ contained in boundary triangle $\Delta_i$. For each $i$, label the angles of $\Delta_i$ as $\alpha_i, \beta_i, \gamma_i$, in clockwise order, such that $\alpha_i$ is the angle cut off by $\sigma_i$. Recall that, in Definition  \ref{def:holonomy}, we defined 
$$\epsilon_i = \left \{
\begin{array}{rl}
1, & \mbox{if $\alpha_i$ is to the left of $\sigma_i$,}\\
-1, & \mbox{if $\alpha_i$ is to the right of $\sigma_i$.}
\end{array} \right.
$$
Comparing this with Definition \ref{def:lead-trail} and Figure \ref{fig:lead-trail}, we see that the vector $\ww(\sigma_i)$ increases angle $\beta_i$ at rate $\epsilon_i$, and increases $\gamma_i$ at rate $-\epsilon_i$.
By Lemma \ref{lemma:vol-deriv},

\begin{eqnarray*}
\frac{\partial \vol}{\partial \ww(\sigma)}
&=& \sum_{i=1}^k \frac{\partial \vol}{\partial \ww(\sigma_i)} \\
&=&  \sum_{i=1}^k \left( - \epsilon_i \log \sin \beta_i + \epsilon_i \log \sin \gamma_i \right) \\
&=& \sum_{i=1}^k \epsilon_i \log \left( \frac{\sin  \gamma_i}{\sin \beta_i} \right) \\
&=& \sum_{i=1}^k \epsilon_i \log \abs{z_i} \qquad \qquad \mbox{by equation (\ref{eq:angle-shape})} \\
&=& Re \: \sum_{i=1}^k \epsilon_i \log z_i ,
\end{eqnarray*}
as desired.
\end{proof}

Recall, from Section \ref{sec:angle-struct}, that an angle structure on $\tau$ corresponds to solving the imaginary part of the edge gluing equations. By Proposition \ref{prop:volume-computation}, solving the real part of each edge equation amounts to having vanishing derivative in the direction of the corresponding deformation. This turns out to be the crucial step in the proof of Theorem \ref{thm:volume-max}.


\begin{named}{Theorem \ref{thm:volume-max}}
Let $M$ be an orientable $3$--manifold with boundary consisting of tori, and let $\tau$ be an ideal triangulation of $M$. Then a point $p \in \ang(\tau)$ corresponds to a complete hyperbolic metric on the interior of $M$ if and only if $p$ is a critical point of the functional $\vol: \ang(\tau) \to \RR$.
\end{named}

\begin{proof}
For one direction of the theorem, suppose that $p \in \ang(\tau)$ is a critical point of $\vol$. This angle structure defines a shape parameter on each tetrahedron of $\tau$. By Proposition \ref{prop:completeness}, proving that these shape parameters give a complete hyperbolic metric on $M$ amounts to checking the edge and completeness equations of Definition \ref{def:gluing}. 

First, consider the edge equation about an edge $e$. Note that the imaginary part of the gluing equation about edge $e$ is automatically satisfied for any angle structure. To check the real part of the gluing equation, let $\sigma \subset \bdy M$ be a normal closed curve encircling one endpoint of $e$. Since $p$ is a critical point of $\vol$, Proposition \ref{prop:volume-computation} implies that $\Ree(H(\sigma))=0$, as desired. Thus the edge gluing equations are satisfied.

To check completeness, let $C$ be a boundary torus of $M$, and let $\sigma_1, \sigma_2$ be a pair of simple closed normal curves that span $\pi_1(C) = H_1(C)$. Recall that $C$ is tiled by boundary triangles that truncate the tips of ideal tetrahedra. The angle structure $p$ gives each of these triangles a Euclidean shape, well-defined up to similarity. If $\Delta$ is a triangle in which $\sigma_1$ and $\sigma_2$ intersect, we may place the corners of $\Delta$ at $0,1,z \in \CC$, and develop the other triangles of $C$ from there. Note that, since the edge gluing equations are satisfied, the boundary triangles fit together correctly around every vertex of $C$.

Let $d_1$ and $d_2$ be the deck transformations of $\widetilde{C}$ corresponding to $\sigma_1$ and $\sigma_2$. Consider the complex numbers
$$q= d_1(0), \qquad r=d_2(0), \qquad s = d_1(r) = d_2 (q).$$
The four points $0,q,r,s$ form four corners of a fundamental domain for $C$.
By the definition of holonomy (Definition \ref{def:holonomy}) and the solution to the edge equations,
$$H(\sigma_1) = \log \frac{s-q}{r-0} \, , \qquad H(\sigma_2) = \log \frac{s-r}{q-0}  \, .$$
By Proposition \ref{prop:volume-computation}, the real part of each of these holonomies is $0$. Thus we have $\abs{s-q}=\abs{r}$ and $\abs{s-r} = \abs{q}$, hence the fundamental domain of $C$ is a parallelogram. Therefore,  $H(\sigma_1) = H(\sigma_2) = 0$, and the completeness equations are satisfied for cusp $C$.

For the converse implication of the theorem, suppose that $p \in \ang(\tau)$ defines a complete metric. Then the tetrahedron shapes corresponding to $p$ must satisfy all the edge and completeness equations of Definition \ref{def:gluing}: we must have $H(\sigma) = 2\pi i$ for every closed curve $\sigma$ encircling an endpoint of an edge, and $H(\sigma_j) = 0$ for a collection of simple closed curves $\sigma_1, \ldots, \sigma_{2k}$ that form a basis of $H_1(\bdy M)$. In particular, the real part of each of these holonomies is $0$. But, by Proposition \ref{prop:tangent-span}, the leading--trailing deformations that correspond to these closed curves span $T_p \ang(\tau)$. Thus $p$ is a critical point of $\vol$, as desired.
\end{proof}

\section{Extensions, applications, generalizations}\label{sec:extensions}

This final section of the paper surveys several ways in which the Casson--Rivin program has been generalized and extended, as well as several infinite families of manifolds to which the program has been successfully applied. 

Theorem \ref{thm:volume-max} concerns manifolds with non-empty boundary that consists of tori. It is natural to ask whether similar methods can be applied to treat manifolds with more general boundary, or closed manifolds that have no boundary at all. Indeed, there has been considerable progress in these areas.

\subsection{Closed manifolds via Dehn filling}\label{sec:filling}
The most straightforward way to extend Theorem \ref{thm:volume-max} to closed manifolds is via Dehn surgery. If $C$ is a boundary torus of $M$, and $ \mu, \lambda$ are  simple closed normal curves that form a basis for $H_1(C)$, then $M(p/q)$ is the manifold obtained by attaching a solid torus to $C$, such that the boundary of the meridian disk is mapped to $p \mu + q \lambda$. In terms of the gluing equations of Definition \ref{def:gluing}, attaching a disk to the closed curve $p \mu + q \lambda$ is equivalent to solving the holonomy equation 
\begin{equation}\label{eq:surgery}
p \, H(\mu) + q \, H(\lambda) = 2 \pi i.
\end{equation}
Just as above, the imaginary part of equation $(\ref{eq:surgery})$ is linear in the angles of $\tau$. Imposing this linear equation corresponds to taking a codimension--$1$ linear slice of the angle space $\ang(\tau)$.

\begin{theorem}\label{thm:angled-surgery}
Let $M$ be a manifold with  boundary a single torus  $C$, and let $\tau$ be an ideal triangulation of $M$. Choose a pair $(p,q)$ of relatively prime integers, and let $\ang_{p/q}(\tau) \subset \ang(\tau)$ be the set of all angle structures that satisfy the imaginary part of equation $(\ref{eq:surgery})$. Then a critical point of the volume functional $\vol$ on $\ang_{p/q}(\tau)$ yields a complete hyperbolic structure on $M(p/q)$, the $p/q$ Dehn filling of $M$.

The analogous statement holds for fillings along multiple boundary tori.
\end{theorem}

The proof follows the same outline as Theorem \ref{thm:volume-max}; see \cite[Theorem 6.2]{chan-hodgson} for more details.

\subsection{Manifolds with polyhedral boundary}\label{sec:polyhedra}
There is also an analogue of Theorem \ref{thm:volume-max} for ideal triangulations where not every face of the tetrahedra is glued to another face. Given such a partial gluing of tetrahedra, one obtains a $3$--manifold $N$ whose boundary is subdivided into ideal triangles. If the tetrahedra carry dihedral angles, then every edge along $\bdy N$ will also carry a prescribed angle. The simplest case of this is when $N$ is an ideal polyhedron, but one may also consider cases where $N$ has more complicated topology.

It is worth asking exactly when a $3$--manifold $N$ with polyhedral boundary, and with a fixed assignment of convex dihedral angles, carries a complete hyperbolic metric 
%
with ideal vertices
that realizes those angles. The following combinatorial condition was suggested by Rivin \cite{rivin:topology}:

\begin{itemize}
\item[$(*)$] For every simple closed normal curve $\sigma \subset \bdy N$ that bounds a disk in $N$, the sum of exterior angles along $\sigma$ is at least $2\pi$, with equality iff $\sigma$ encircles an ideal vertex.
\end{itemize}

\begin{theorem}\label{thm:polyhedra}
Let $N$ be a $3$--manifold with polyhedral boundary and prescribed (convex) dihedral angles along every edge of $\bdy N$. Suppose that $N$ is irreducible and atoroidal, and furthermore that $N$ is a $3$--ball, or a solid torus, or has incompressible boundary. Then $N$ carries a complete hyperbolic metric realizing the prescribed angles if and only if condition $(*)$ holds. Furthermore, any hyperbolic realization is unique up to isometry.
\end{theorem}

%
The case when $N$ is a polyhedron is due to Rivin \cite{rivin:polyhedron, rivin:optimization}; the case of solid tori, to Gu\'eritaud \cite{gueritaud:solid-torus}. In both cases, the argument works by first proving that $N$ has an ideal triangulation with a non-empty angle space $\ang(\tau)$, and then proving that the volume functional $\vol$ has a critical point in $\ang(\tau)$. This critical point gives the hyperbolic realization of $N$. In the preprint \cite{schlenker:polyhedral-boundary}, Schlenker gives a more analytic and general argument in the case of incompressible boundary. Conjecturally, the special hypotheses of all these papers are not needed: all that should be necessary is that $N$ is irreducible and contains no incompressible tori, and that the angle assignments satisfy $(*)$. See \cite[Conjecture 2.4]{fg-arborescent}.

\subsection{Generalized angle structures}\label{sec:circle-valued}

All of the angle structures discussed so far have involved strictly positive dihedral angles. One natural generalization of the definitions in Section \ref{sec:angle-struct} would be to allow negative angles, or more generally, to consider angles mod $2\pi$. 

A \emph{generalized angle structure} is an assignment of a real number to every pair of opposite edges in a tetrahedron, so that equations $(2)$ and $(3)$ of Definition \ref{def:angled-polytope} are satisfied, but the inequalities are discarded. The set of all such assignments is denoted  $\mathcal{GA}(\tau)$. Luo and Tillmann showed that for any ideal triangulation $\tau$ of a manifold with torus boundary, $\mathcal{GA}(\tau)$ is always non-empty \cite{luo-tillmann:normal-surf-angle}; in other words, there is no analogue of Theorem \ref{thm:ang-hyperbolization}. Furthermore, they establish a linear--algebraic duality between angle structures and normal surfaces
%
(see also Rivin \cite{rivin:optimization}).
For example, the obstruction to finding a non-empty positive polytope $\ang(\tau) \subset \mathcal{GA}(\tau)$ is a certain branched normal surface with non-negative Euler characteristic.

Generalizing further, an \emph{$S^1$--valued angle structure} on a triangulation $\tau$ is an assignment of a real number (mod $2\pi$) to every pair of opposite edges in a tetrahedron, such that
\begin{enumerate}
\item Around each ideal vertex of a tetrahedron, the dihedral angles sum to $\pi$ (mod $2\pi$),
\item Around each edge of $M$, the dihedral angles sum to $0$ (mod $2\pi$).
\end{enumerate}
The set of all $S^1$--valued angle structures on $\tau$ is denoted $\circang(\tau)$.

Just as with real--valued solutions in $\mathcal{GA}(\tau)$, Luo showed that the existence of $S^1$--valued solutions is extremely general. For any triangulated closed pseudo--manifold (i.e., any cell complex obtained by gluing tetrahedra in pairs along all of their faces, whatever the link of a vertex), he showed that $\circang(\tau) \neq \emptyset$, and is a closed smooth manifold \cite[Proposition 2.6]{luo:gen-thurston}.

Even though the existence of $S^1$--valued angle structures does not distinguish the class of hyperbolic manifolds, studying the volume of such a structure can still yield geometric information. Recall from Lemma \ref{lemma:tet-volume} that the Lobachevsky function $\lob$ is $\pi$--periodic; as a result, Definition \ref{def:vol-functional} of the volume functional $\vol$ extends in a natural way to $\circang(\tau)$. Because $\circang(\tau) \neq \emptyset$ is a compact manifold, $\vol$ must achieve a maximum. The work of Luo \cite{luo:gen-thurston} and Luo--Tillmann \cite{luo-tillmann:gen-volume} uses this maximum point to either solve a generalized version of the gluing equations (which yields a representation from a double branched cover of $M$ into $PSL(2, \CC)$), or find certain highly restrictive normal surfaces in $M$. See Luo's survey paper in this volume \cite{luo:survey} for more details.

\subsection{Volume estimates}\label{sec:vol-estimate}
Recall that, by Lemma \ref{lemma:vol-deriv}, the volume functional $\vol$ is concave down on $\ang(\tau)$. As a result, any critical point of $\vol$ must actually be the global maximum of the function over the compact closure $\overline{\ang(\tau)}$. Thus Theorem \ref{thm:volume-max} has the following corollary.

\begin{theorem}\label{thm:vol-estimate}
Let $M$ be an orientable $3$--manifold with boundary consisting of tori, and let $\tau$ be an ideal triangulation of $M$. Suppose that $\vol: \ang(\tau) \to \RR$ has a critical point at $p \in \ang(\tau)$. Then, for any point $q \in \overline{\ang(\tau)}$,
$$\vol(q) \leq \mathrm{vol}(M),$$
with equality iff $q=p$ (i.e., iff $q$ gives the complete hyperbolic metric on $M$).
\end{theorem}

In fact, the analogous statement also holds in the settings of Section \ref{sec:filling} (Dehn filling) or Section \ref{sec:polyhedra} (polyhedral boundary). The uniqueness of a critical point of $\vol$ turns out to be the key idea in Rivin's proof of the uniqueness statement of Theorem \ref{thm:polyhedra}.


Theorem \ref{thm:vol-estimate} 
%
allows to compute
effective, combinatorial volume estimates for hyperbolic $3$--manifolds. In certain settings, the combinatorics of a $3$--manifold naturally guides a choice of triangulation $\tau$, and the same combinatorial data provide a convenient point $q \in \overline{\ang(\tau)}$. Then, Theorem \ref{thm:vol-estimate} says that $\vol(q)$ is a lower bound on the volume of $M$. 
%
This approach
is illustrated by the first family of manifolds to which the Casson--Rivin method was successfully applied \cite{gf:bundles}.

A decade after Colin de Verdi{\`e}re, Casson, and Rivin developed the theory of volume maximization in the early 1990s, it was Gu\'eritaud who first applied the method to find the hyperbolic metrics on an infinite family of manifolds, namely punctured torus bundles  \cite{gf:bundles}. In an appendix to the same paper, Futer extended the method to two--bridge links \cite[Appendix]{gf:bundles}. For both of these families of manifolds, the existence of hyperbolic metrics 
%
was well-known,
but the volume estimates coming from Theorem \ref{thm:vol-estimate} were both new and sharp. Combined with Dehn surgery techniques, the volume estimates from   \cite{gf:bundles} also give explicit, combinatorial bounds for the volume of several families of knot and link complements \cite{fkp:coils, fkp:farey}, as well as of a number of closed manifolds \cite{petronio-vesnin}.

It is worth asking whether the existence of a critical point of $\vol$ in $\ang(\tau)$ is actually necessary for the volume inequality of Theorem \ref{thm:vol-estimate}. Casson conjectured that this inequality holds for any angled triangulation $\tau$, whether or not the tetrahedra of this triangulation can be given positively oriented shapes in the hyperbolic metric on $M$. (Note that, by Theorem \ref{thm:ang-hyperbolization}, the hyperbolic metric must exist.) If proved true, this conjecture would provide a practical tool for finding volume estimates on many more families of $3$--manifolds.

\subsection{Canonical triangulations}\label{sec:canonical}
Nearly everything discussed thus far in this paper has depended on the choice of triangulation. As it turns out, many of the methods already discussed can show that a particular triangulation is \emph{geometrically canonical} for $M$.

Given a hyperbolic $3$--manifold $M$ with $k$ cusps, let $H_1, \ldots, H_k$ be disjoint horospherical neighborhoods of the cusps. Then the Ford--Voronoi domain $\mathcal{F}$ is the set of all points in $M$ that have a unique shortest path to the union of the $H_i$. This is an open set in $M$, whose complement $L = M \setminus \mathcal{F}$ is a compact $2$--complex, called the \emph{cut locus}. The dual to $L$ is an ideal polyhedral decomposition $\mathcal{P}$ of $M$; the $n$--cells of $\mathcal{P}$ are in bijective correspondence with the $(3-n)$--cells of $L$. This is called the \emph{canonical polyhedral decomposition} of $M$, relative to the cusp neighborhoods $H_i$.

The combinatorics of $L$ ---and therefore, of $\mathcal{P}$--- depends only on the relative volumes of the $H_i$. In particular, if $M$ has only one cusp, there are no choices whatsoever, and $\mathcal{P}$ is completely determined by the hyperbolic metric. If the horoballs in $\HH^3$ obtained by lifting the cusp neighborhoods $H_i$ are in ``general position,'' every vertex of $\widetilde{L}$ will have exactly four closest horoballs.  Thus each vertex of $L$ will meet four edges, and the dual  polyhedral decomposition $\mathcal{P}$ will generically be a triangulation.

If $\tau$ is a given triangulation (or, more generally, a given polyhedral decomposition), proving that $\tau$ is canonical amounts to verifying finitely many inequalities about nearest horoballs. Epstein and Penner found a way to translate these inequalities into convexity statements in the Minkowski space $\RR^{3+1}$, where $\HH^3$ is modeled by a hyperboloid \cite{epstein-penner}. More recently, Gu\'eritaud discovered that 
%
%
the convexity inequalities that imply canonicity (once translated to inequalities involving dihedral angles of the polyhedra) can be verified using information obtained in the course of showing that $\vol: \ang(\tau) \to \RR$ has a critical point \cite{gueritaud:thesis}.

To date, the method of angled triangulations has found both the hyperbolic metric and the canonical polyhedral decomposition of several families of $3$--manifolds: punctured--torus bundles \cite[Theorem 1.11.1]{gueritaud:thesis}, two-bridge links \cite[Theorem 2.1.7]{gueritaud:thesis}, and certain special arborescent links \cite[Theorem 2.3.1]{gueritaud:thesis}. 
%
Gu\'eritaud used the same ideas to find the canonical triangulations for convex cores of quasi-Fuchsian punctured--torus groups \cite{gueritaud:quasifuchsian}. All of these 
%
%
families of manifolds actually have closely related combinatorial features, with the structure of the triangulation effectively determined by the combinatorics of $SL(2,\ZZ)$ and continued fractions.

There is a recent result that brings together several themes from this section. Suppose that $M$ is a hyperbolic $3$--manifold with $k$ cusps, and the canonical decomposition of $M$ is indeed a triangulation. If we perform Dehn filling along one of these cusps, the result will (generically) be another hyperbolic manifold $M(p/q)$, with a new geometry. In 
%
\cite{gueritaud-schleimer},
Gu\'eritaud and Schleimer use the canonical triangulation of $M$ to completely describe the canonical triangulation of its generic Dehn fillings. Their argument uses Theorems~\ref{thm:angled-surgery} and \ref{thm:polyhedra} to construct a triangulated solid torus with 
%
the right shape and glue it into $M(p/q)$.
%


\subsection{Weil rigidity}\label{sec:weil}

We close this paper with an application whose statement has nothing to do with triangulations. The following rigidity theorem, 
due to Weil \cite{weil:rigidity}, precedes Mostow--Prasad rigidity by a dozen years. 

\begin{theorem}\label{thm:weil-rigidity}
Let $M$ be a $3$--manifold with boundary consisting of tori. Then any complete hyperbolic metric on the interior of $M$ is locally rigid: there is no local deformation of the metric through other complete hyperbolic metrics.
\end{theorem}

We thank Marc Culler and Feng Luo for a fruitful discussion that produced the following extremely short proof.

\begin{proof}
Suppose that $M$ admits a complete hyperbolic metric. Choose horoball neighborhoods $H_1, \ldots, H_k$ about the cusps of $M$. Then, as in Section \ref{sec:canonical}, this choice of cusp neighborhoods determines a decomposition $\mathcal{P}$ of $M$ into ideal polyhedra. Luo, Schleimer, and Tillmann showed that $M$ has a finite--sheeted cover $N$, in which the lift of $\mathcal{P}$ decomposes into positively oriented ideal tetrahedra \cite{lst:geodesic-tetrahedra}. Let $\tau$ be this positively oriented ideal triangulation of $N$. 

By Theorem \ref{thm:volume-max}, the complete hyperbolic metric on $N$ (which was obtained by lifting the metric on $M$) represents a critical point of $\vol: \ang(\tau) \to \RR$. By Lemma \ref{lemma:vol-deriv}, this critical point is unique. But any local deformation of the complete metric on $M$ would lift to a deformation of the metric on $N$, which would violate Lemma \ref{lemma:vol-deriv}.
\end{proof}


\bibliographystyle{hamsplain}
\bibliography{biblio.bib}

\end{document}